\documentclass{amsart}
\usepackage{amssymb,amsmath,latexsym,amsfonts,amsthm,mathrsfs,lineno}
\usepackage{enumerate}

\usepackage[utf8]{inputenc}
\usepackage[T1]{fontenc}
\usepackage[left=3.5cm,right=3.5cm,top=3cm,bottom=3cm]{geometry}
\usepackage{lmodern}
\usepackage[T1]{fontenc}
\usepackage{tikz}

\usepackage{hyperref}
\usepackage{graphicx}
\usepackage{color}

\newtheorem{thm}{Theorem}
\newtheorem*{thm*}{Theorem}
\newtheorem*{defn}{Definition}
\newtheorem{prop}[thm]{Proposition}

\newtheorem{lemma}[thm]{Lemma}

\newcommand{\N}{\mathbb{N}}

\newcommand{\s}{\mathbb{S}}

\newcommand{\p}{\mathcal{P}}

\newcommand{\Aut}{\mathrm{Aut}}

\newtheorem{theorem}{Theorem}

\author{Colin Jahel}
\address{Université Paris Diderot, Institut de Mathématiques de Jussieu-Paris Rive Gauche}
\email{colin.jahel@imj-prg.fr}
\thanks{Research was partially supported by the ANR project AGRUME
(ANR-17-CE40-0026).}
\date{January 2019}
\title[Unique ergo. of the automorphism group of the semigen. directed graph]{Unique ergodicity of the automorphism group of the semigeneric directed graph}

\begin{document}

\begin{abstract}
We prove that the automorphism group of the semigeneric directed graph (in the sense of Cherlin's classification) is uniquely ergodic.
\end{abstract}
\maketitle

\section{Introduction}

One key notion in the study of dynamical properties of Polish groups is amenability. A topological group is amenable when every flow, i.e. continuous action on a compact space, admits a Borel probability measure that is invariant under the action of the group.

In recent years, the study of non-locally compact Polish groups has exhibited several refinements of this phenomenon. One of them is extreme amenability: a topological group is extremely amenable when every flow admits a fixed point (see \cite{KPT}). Another one is unique ergodicity: a topological group is uniquely ergodic if every minimal flow, i.e. a flow where every orbit is dense, admits a unique Borel probability measure that is invariant under the action of the group. In this paper, all measures will be Borel probability measures.

Of course, extreme amenability implies unique ergodicity, but the converse is not true as for instance, every compact group is uniquely ergodic. Beyond compactness, though, no example is known in the locally compact Polish case and Weiss proves in \cite{Weiss} that there is no uniquely ergodic discrete group. In fact, it is suggested on page $5$ in \cite{AKL} that in the setting of locally compact groups, compactness is the only way to achieve unique ergodicity. However, some examples appear in the non-locally compact Polish case. The first of these examples was $S_\infty$, the group of all permutations of $\N$ equipped with the pointwise convergence topology (this was done by Glasner and Weiss in \cite{GW}). Angel, Kechris and Lyons then showed, using probabilistic combinatorial methods, that several groups of the form $\Aut(\mathbb{F})$, where $\mathbb{F}$ is a particular kind of countable structure called Fraïssé limit, are also uniquely ergodic (see \cite{AKL}).

A Fraïssé limit is a countable first-order homogeneous structure in the sense of model theory whose age, i.e. the set of its finite substructures up to isomorphism, is a Fraïssé class. A class $\mathcal{F}$ of finite structures is a Fraïssé class if it contains structures of arbitrarily large (finite) cardinality and satisfies the following:
\begin{itemize}
\item[(HP)]If $A \in \mathcal{F}$ and $B$ is a substructure of $A$, then $B\in \mathcal{F}$.
\item[(JEP)] If $A,B\in \mathcal{F}$ then there exists $C\in \mathcal{F}$ such that $A$ and $B$ can be embedded in $C$.
\item[(AP)]If $A, B, C \in \mathcal{F}$ and $f \colon A\rightarrow B$, $g \colon A\rightarrow C$ are embeddings, then there exists $D\in \mathcal{F}$ and $h \colon B\rightarrow D$, $l \colon C\rightarrow D$ embeddings such that $h\circ f = l \circ g$.
\end{itemize}

Examples of Fraïssé classes include the class of finite graphs, the class of finite graphs omitting a fixed clique, the class of finite $r$-uniform hypergraphs for any $r\in \N$. The unique ergodicity of the automorphism groups of the limits of those classes was proven in \cite{AKL}. 

The Fraïssé limit of a Fraïssé class is unique up to isomorphism. By definition, Fraïssé limits are homogeneous, i.e. any isomorphism between two finite parts of the structure can be extended in an automorphism of the structure. For more details on Fraïssé classes see \cite{Hodges}.

In \cite{SP}, using methods from \cite{AKL}, Pawliuk and Soki\'c extended the catalogue of uniquely ergodic automorphism groups with the automorphism groups of homogeneous directed graphs, which were all classified by Cherlin (see \cite{Cher}), leaving as an open question only the case of the semigeneric directed graph.

This graph, which we denote $\s$, is the Fraïssé limit of the class $\mathcal{S}$ of simple, loopless, directed, finite graphs that verify the following conditions:
\begin{itemize}
\item[$i)$] the relation $\perp$, defined by $x\perp y $ iff $\neg(x\rightarrow y \vee y \rightarrow x)$, is an equivalence relation,
\item[$ii)$] for any $x_1\neq x_2,y_1\neq y_2$ such that $x_1\perp x_2$ and $y_1\perp y_2$, the number of (directed) edges from $\{x_1,x_2\}$ to $\{y_1,y_2\}$ is even,
\end{itemize}
where $\rightarrow$ denotes the directed edge. We will refer to $\perp$-equivalence classes as columns and to the second condition as the parity condition. The $\perp$-class of an element $a\in \s$ will be referred to as $a^\perp$.

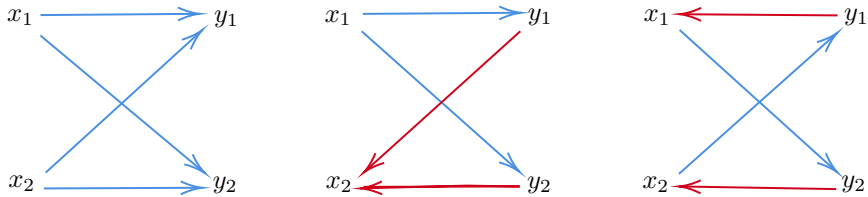
\begin{figure}[!h]
\tikzset{every picture/.style={line width=0.75pt}} 

\begin{tikzpicture}[x=0.75pt,y=0.75pt,yscale=-1,xscale=1]

\draw [color={rgb, 255:red, 74; green, 144; blue, 226 }  ,draw opacity=1 ]  (78.17,18.67) -- (156.43,18.05) ;
\draw [shift={(158.43,18.03)}, rotate = 539.55] [color={rgb, 255:red, 74; green, 144; blue, 226 }  ,draw opacity=1 ][line width=0.75]    (10.93,-3.29) .. controls (6.95,-1.4) and (3.31,-0.3) .. (0,0) .. controls (3.31,0.3) and (6.95,1.4) .. (10.93,3.29)   ;
\draw [color={rgb, 255:red, 74; green, 144; blue, 226 }  ,draw opacity=1 ]  (79.83,106.17) -- (157.83,105.84) ;
\draw [shift={(159.83,105.83)}, rotate = 539.76] [color={rgb, 255:red, 74; green, 144; blue, 226 }  ,draw opacity=1 ][line width=0.75]    (10.93,-3.29) .. controls (6.95,-1.4) and (3.31,-0.3) .. (0,0) .. controls (3.31,0.3) and (6.95,1.4) .. (10.93,3.29)   ;
\draw [color={rgb, 255:red, 74; green, 144; blue, 226}  ,draw opacity=1 ] (240.73,18.07) -- (318.5,17.51) ;
\draw [shift={(320.5,17.5)}, rotate = 539.5899999999999] [color={rgb, 255:red, 74; green, 144; blue, 226 }  ,draw opacity=1 ][line width=0.75]    (10.93,-3.29) .. controls (6.95,-1.4) and (3.31,-0.3) .. (0,0) .. controls (3.31,0.3) and (6.95,1.4) .. (10.93,3.29)   ;
\draw [color={rgb, 255:red, 208; green, 2; blue, 27 }  ,draw opacity=1 ]   (241.1,105.37) -- (319.83,105.17) -- (293.07,104.74) -- (241.5,106.11) ;
\draw [shift={(239.5,106.17)}, rotate = 358.48] [color={rgb, 255:red, 208; green, 2; blue, 27 }  ,draw opacity=1 ][line width=0.75]    (10.93,-3.29) .. controls (6.95,-1.4) and (3.31,-0.3) .. (0,0) .. controls (3.31,0.3) and (6.95,1.4) .. (10.93,3.29)   ;
\draw [color={rgb, 255:red, 208; green, 2; blue, 27 }  ,draw opacity=1 ]   (480.17,18.83) -- (402.43,18.25) ;
\draw [shift={(400.43,18.23)}, rotate = 360.43] [color={rgb, 255:red, 208; green, 2; blue, 27 }  ,draw opacity=1 ][line width=0.75]    (10.93,-3.29) .. controls (6.95,-1.4) and (3.31,-0.3) .. (0,0) .. controls (3.31,0.3) and (6.95,1.4) .. (10.93,3.29)   ;
\draw [color={rgb, 255:red, 74; green, 144; blue, 226 }  ,draw opacity=1 ] (78.17,29) -- (159.63,97.22) ;
\draw [shift={(161.17,98.5)}, rotate = 219.94] [color={rgb, 255:red, 74; green, 144; blue, 226 }  ,draw opacity=1 ][line width=0.75]    (10.93,-3.29) .. controls (6.95,-1.4) and (3.31,-0.3) .. (0,0) .. controls (3.31,0.3) and (6.95,1.4) .. (10.93,3.29)   ;
\draw [color={rgb, 255:red, 74; green, 144; blue, 226 }  ,draw opacity=1 ]  (80.83,98.17) -- (158.36,27.51) ;
\draw [shift={(159.83,26.17)}, rotate = 497.65] [color={rgb, 255:red, 74; green, 144; blue, 226 }  ,draw opacity=1 ][line width=0.75]    (10.93,-3.29) .. controls (6.95,-1.4) and (3.31,-0.3) .. (0,0) .. controls (3.31,0.3) and (6.95,1.4) .. (10.93,3.29)   ;
\draw[color={rgb, 255:red, 74; green, 144; blue, 226 }  ,draw opacity=1 ]   (240.17,26.83) -- (318.34,96.5) ;
\draw [shift={(319.83,97.83)}, rotate = 221.71] [color={rgb, 255:red, 74; green, 144; blue, 226 }  ,draw opacity=1 ][line width=0.75]    (10.93,-3.29) .. controls (6.95,-1.4) and (3.31,-0.3) .. (0,0) .. controls (3.31,0.3) and (6.95,1.4) .. (10.93,3.29)   ;
\draw [color={rgb, 255:red, 208; green, 2; blue, 27 }  ,draw opacity=1 ]   (320.17,26.83) -- (242.65,96.83) ;
\draw [shift={(241.17,98.17)}, rotate = 317.91999999999996] [color={rgb, 255:red, 208; green, 2; blue, 27 }  ,draw opacity=1 ][line width=0.75]    (10.93,-3.29) .. controls (6.95,-1.4) and (3.31,-0.3) .. (0,0) .. controls (3.31,0.3) and (6.95,1.4) .. (10.93,3.29)   ;
\draw [color={rgb, 255:red, 74; green, 144; blue, 226 }  ,draw opacity=1 ]  (400.17,98.83) -- (478.64,28.86) ;
\draw [shift={(480.13,27.53)}, rotate = 498.28] [color={rgb, 255:red, 74; green, 144; blue, 226 }  ,draw opacity=1 ][line width=0.75]    (10.93,-3.29) .. controls (6.95,-1.4) and (3.31,-0.3) .. (0,0) .. controls (3.31,0.3) and (6.95,1.4) .. (10.93,3.29)   ;
\draw [color={rgb, 255:red, 74; green, 144; blue, 226 }  ,draw opacity=1 ]   (400.5,26.17) -- (478.68,96.83) ;
\draw [shift={(480.17,98.17)}, rotate = 222.11] [color={rgb, 255:red, 74; green, 144; blue, 226 }  ,draw opacity=1 ][line width=0.75]    (10.93,-3.29) .. controls (6.95,-1.4) and (3.31,-0.3) .. (0,0) .. controls (3.31,0.3) and (6.95,1.4) .. (10.93,3.29)   ;
\draw [color={rgb, 255:red, 208; green, 2; blue, 27 }  ,draw opacity=1 ]   (479.5,106.5) -- (402.17,105.2) ;
\draw [shift={(400.17,105.17)}, rotate = 360.96000000000004] [color={rgb, 255:red, 208; green, 2; blue, 27 }  ,draw opacity=1 ][line width=0.75]    (10.93,-3.29) .. controls (6.95,-1.4) and (3.31,-0.3) .. (0,0) .. controls (3.31,0.3) and (6.95,1.4) .. (10.93,3.29)   ;

\draw (59.93,14.6) node [anchor=north west][inner sep=0.75pt]    {$x_{1}$};
\draw (220.14,14.73) node [anchor=north west][inner sep=0.75pt]    {$x_{1}$};
\draw (380.93,14.4) node [anchor=north west][inner sep=0.75pt]    {$x_{1}$};
\draw (60.6,98.27) node [anchor=north west][inner sep=0.75pt]    {$x_{2}$};
\draw (220.6,98.6) node [anchor=north west][inner sep=0.75pt]    {$x_{2}$};
\draw (380.27,98.6) node [anchor=north west][inner sep=0.75pt]    {$x_{2}$};
\draw (164.27,14.93) node [anchor=north west][inner sep=0.75pt]    {$y_{1}$};
\draw (322.43,14.1) node [anchor=north west][inner sep=0.75pt]    {$y_{1}$};
\draw (481.97,14) node [anchor=north west][inner sep=0.75pt]    {$y_{1}$};
\draw (163.6,98.6) node [anchor=north west][inner sep=0.75pt]    {$y_{2}$};
\draw (322.17,98.57) node [anchor=north west][inner sep=0.75pt]    {$y_{2}$};
\draw (480.6,98.27) node [anchor=north west][inner sep=0.75pt]    {$y_{2}$};
\end{tikzpicture}
\caption{The three possible configurations (up to isomorphism) of two pairs of equivalent points respecting the parity condition.}

\end{figure}


More details on this structure will be given in the next section.

In this paper, we prove:
\begin{thm}\label{thm1} The topological group $\Aut (\s)$ is uniquely ergodic.
\end{thm}

The method we use is different from the one found in \cite{AKL} and \cite{SP} since we do not work with the so-called "quantitative expansion property", but rather show that an ergodic measure can only take certain values on a generating part of the Borel sets. It is also different from the approach in \cite{Tsan} (see Theorem $7.4$) which only applies when the structure eliminates imaginaries. Our method relies on the idea that if there are equivalence classes in a structure and the universal minimal flow is essentially the convex orderings regarding the equivalence classes, then the ordering inside the equivalence classes and the ordering of the equivalence classes are independent, provided that the automorphism group behaves well enough.
\\

\emph{Acknowledgements}:

I am grateful to my PhD supervisors Lionel Nguyen Van Thé and Todor Tsankov for their helpful advice during my research on this paper. I also want to thank Miodrag Soki\'c for his comments on this paper. I thank the referee, whose comments helped me greatly improve the structure of the paper.

\section{Preliminaries}

The starting point of our proof is common with that of \cite{AKL}: to prove that $\Aut(\s)$ is uniquely ergodic, it suffices to show that one particular action is uniquely ergodic, namely, its universal minimal flow, $\Aut(\s) \curvearrowright \mathrm{M}\left(\Aut(\s)\right)$. This is the unique minimal $\Aut(\s)$-flow that maps onto any minimal $\Aut(\s)$-flow (such a flow exists for any Hausdorff topological group by a classical result of Ellis, see \cite{Ellis}); an explicit description was made by Jasi\'nski, Laflamme,  Nguyen Van Th\'e and Woodrow in \cite{JLNVTW}. It is the space of expansions of $\s$ whose $\mathrm{Age}$ is a certain class $\mathcal{S}^*$.

Before describing this class, we give some more background on $\s$. Observe that the parity condition is equivalent to the fact for every $\mathbf{A}\in \mathcal{S}$ and two columns $P,Q$ in $\mathbf{A}\in \mathcal{S}$, we have for all $x,x'\in P$,
\[\left(
\forall y\in Q \ \left((x\to y)\Leftrightarrow(x'\to y) \right)
\right)
\text{ or }
\left(
\forall y\in Q \ \left((x\to y) \Leftrightarrow (y\to x')\right)
\right).\]

This remark allows us to define the equivalence relation $\sim_Q$ on $P$ as:
\[
x\sim_Q x' \Leftrightarrow \ \forall y\in Q \ (x\to y \Leftrightarrow x' \to y).
\]

Note that as a consequence of the parity condition, we get that in $\s$, 
\[ 
\forall y\in Q \ (x\to y \Leftrightarrow x' \to y) \Leftrightarrow  \ \exists y \in Q \quad (x\to y \text{ and } x' \to y).
\] 

We can now consider $P^0$ and $P^1$ the two $\sim_Q$ equivalence classes in $P$, and we have $P=P^0\sqcup P^1$. Note that each of these class could be empty. Similarly, we have  $Q=Q^0 \sqcup Q^1$, where $Q^0$ and $Q^1$ are $\sim_P$-equivalence classes. Note that at that stage, this labelling of these classes is arbitrary, which is crucial to the construction and understanding of $\mathcal{S}^*$ bellow. Indeed, the language of $\mathcal{S}^*$ has a binary relation $R$ which interpretation is mainly to give a proper labelling of those equivalence classes.

This description has an interesting consequence when we recall that there must be an edge between any two points of $P$ and $Q$. Denote $P^i\to Q^j$ to mean for all $ x\in P^i$ and $ y \in Q^j$, we have $ x\to y$. Then $P^i \to Q^j$, implies that $Q^j \to P^{1-i}, \ P^{1-i} \to Q^{1-j}$ and $ Q^{1-j} \to P^i.$ In particular, this means that for each $i\in \{0,1\}$, there is a unique $j\in \{0,1\}$ such that $P^i \to Q^j$.

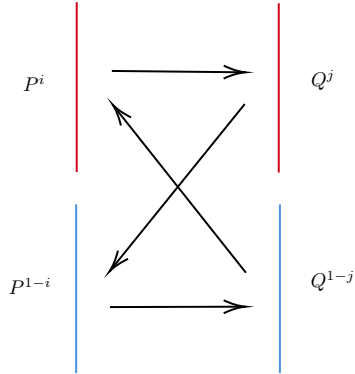
\begin{figure}[!h]

\tikzset{every picture/.style={line width=0.75pt}} 

\begin{tikzpicture}[x=0.75pt,y=0.75pt,yscale=-1,xscale=1]

\draw [color={rgb, 255:red, 208; green, 2; blue, 27 }  ,draw opacity=1 ]   (402.14,38.36) -- (402.14,123.79) ;
[color={rgb, 255:red, 74; green, 144; blue, 226 }  ,draw opacity=1 ]  (402.71,139.79) -- (402.71,224.64) ;
\draw [color={rgb, 255:red, 208; green, 2; blue, 27 }  ,draw opacity=1 ]   (300.24,37.81) -- (300.24,123.52) ;
\draw  [color={rgb, 255:red, 74; green, 144; blue, 226 }  ,draw opacity=1 ] (299.95,139.88) -- (299.95,225.02) ;
\draw  [color={rgb, 255:red, 74; green, 144; blue, 226 }  ,draw opacity=1 ] (402.71,139.88) -- (402.71,225.02) ;
\draw    (385.57,174.36) -- (319.55,91.02) ;
\draw [shift={(318.31,89.45)}, rotate = 411.61] [color={rgb, 255:red, 0; green, 0; blue, 0 }  ][line width=0.75]    (10.93,-3.29) .. controls (6.95,-1.4) and (3.31,-0.3) .. (0,0) .. controls (3.31,0.3) and (6.95,1.4) .. (10.93,3.29)   ;
\draw    (317.95,72.52) -- (383.29,73.19) ;
\draw [shift={(385.29,73.21)}, rotate = 180.59] [color={rgb, 255:red, 0; green, 0; blue, 0 }  ][line width=0.75]    (10.93,-3.29) .. controls (6.95,-1.4) and (3.31,-0.3) .. (0,0) .. controls (3.31,0.3) and (6.95,1.4) .. (10.93,3.29)   ;
\draw    (316.99,191.7) -- (383.29,191.51) ;
\draw [shift={(385.29,191.5)}, rotate = 539.8299999999999] [color={rgb, 255:red, 0; green, 0; blue, 0 }  ][line width=0.75]    (10.93,-3.29) .. controls (6.95,-1.4) and (3.31,-0.3) .. (0,0) .. controls (3.31,0.3) and (6.95,1.4) .. (10.93,3.29)   ;
\draw    (385,89.21) -- (317.97,172.23) ;
\draw [shift={(316.71,173.79)}, rotate = 308.91999999999996] [color={rgb, 255:red, 0; green, 0; blue, 0 }  ][line width=0.75]    (10.93,-3.29) .. controls (6.95,-1.4) and (3.31,-0.3) .. (0,0) .. controls (3.31,0.3) and (6.95,1.4) .. (10.93,3.29)   ;

\draw (272.12,73.35) node [anchor=north west][inner sep=0.75pt]  [xscale=0.75,yscale=0.75]  {$P^{i}$};
\draw (264.93,175.69) node [anchor=north west][inner sep=0.75pt]  [xscale=0.75,yscale=0.75]  {$P^{1-i}$};
\draw (417.3,70.6) node [anchor=north west][inner sep=0.75pt]  [xscale=0.75,yscale=0.75]  {$Q^{j}$};
\draw (417.13,172.34) node [anchor=north west][inner sep=0.75pt]  [xscale=0.75,yscale=0.75]  {$Q^{1-j}$};

\end{tikzpicture}

\end{figure}

The class $\mathcal{S}^*$ is the class of finite structures in the language $L=(\rightarrow, <, R)$, verifying :
\begin{enumerate}
\item
 $\mathcal{S}^*_{|\rightarrow} = \mathcal{S}$, 
 
 \item $<$ is interpreted as a linear ordering convex with respect to the columns, i.e. the columns are intervals for the ordering. For two columns $P, Q$, we will therefore write $P<Q$ to mean that for all $x\in P,y\in Q$ we have $x<y$. 
 
\item 
For $\mathbf{A}^*\in \mathcal{S}^*$, the binary relation $R^{\mathbf{A}^*}$ verifies

\begin{itemize} 
\item[(a)] For all $x$ and $y\perp y'$, $R^{\mathbf{A}^*}(x,y) \Leftrightarrow R^{\mathbf{A}^*}(x,y')$.
\item[(b)] For $P,$ $Q$ any two columns of $\mathbf{A}^*$ with $P<Q$, there is a unique $\sim_Q$-equivalence class $u$ (possibly empty) in $P$ such that
\[
\forall x\in P,  y \in  Q \quad R^{\mathbf{A}^*}(x,y) \Leftrightarrow x\in u.
\] 
\item[(c)] For $P,$ $Q$ any two columns of $\mathbf{A}^*$ with $P<Q$, if there is $x_1\in P$ such that for all $y\in Q$, $R^{\mathbf{A}^*}(x_1,y)$, then
\[ \forall y\in Q, x \in P \quad R^{\mathbf{A}^*}(y,x) \Leftrightarrow y \to x_1.\]
And if there is no such $x_1$ then we have
\[ \forall y\in Q, x \in P \quad \neg R^{\mathbf{A}^*}(y,x).\]
\item[(d)]If $x\perp^{\mathbf{A}^*} y$, we have $\neg R^{\mathbf{A}^*}(x,y)$.
\end{itemize}

\end{enumerate}

Observe that in a structure $\mathbf{A}^*\in \mathcal{S}^*$, $R^{\mathbf{A}^*}$ gives us a proper labelling of the $\sim_Q$-equivalence classes in $P$ when $P<Q$. In particular, we can render the arbitrary decomposition $P=P^0\sqcup P^1$, $Q=Q^0\sqcup Q^1$ canonical by setting
\[x\in P^1 \Leftrightarrow (\forall y \in Q \quad R^{\mathbf{A}^*}(x,y)) \]
and 
\[y\in Q^1 \Leftrightarrow (\forall x \in P \quad R^{\mathbf{A}^*}(y,x)). \]




A remarkable property of this decomposition is that the edge relation is actually entirely defined by it. Indeed, take two columns $P,Q$ in $\mathbf{A}^*$ that we decompose as above in $P=P^0 \sqcup P^1$, $Q=Q^0\sqcup Q^1$. We know, by construction of $R$ on $Q$, that $Q^1 \to P^1$. As we observed before, this means that $P^1\to Q^0$, $P^0 \to Q^1$ and $Q^0 \to P^0$.

Another point of view on this expansion is given in \cite{JLNVTW}. Take $\mathbf{A}\in \mathcal{S}$ with $n$ columns $P_1,\ldots, P_n$ and an expansion $\mathbf{A}^*\in \mathcal{S}^*$. The expansion $\mathbf{A}^*$ is interdefinable with a structure $\mathbf{A}^{**}$ in the language $\{\to, <,L_{i,f}\}$ where $L_{i,f}$ is a unary predicate for all $i\in \{1,\ldots,n\}=[n]$ and $f\in 2^{[n]\backslash i}$. We have $\mathbf{A}^*_{|\to, <}= \mathbf{A}^{**}_{|\to, <}$. Assuming that $P_1<^{\mathbf{A}^*}\ldots <^{\mathbf{A}^*} P_n$, then we define
\[L^{\mathbf{A}^{**}}_{i,f}=\{x\in P_i \ \colon \ \forall j\in [n]\backslash i, y\in P_j \quad (f(j)=1 \Leftrightarrow  R^{\mathbf{A}^*}(x,y) \} .\]

Denote $\mathcal{M}\subset \{0,1\}^{\s^2} \times \{0,1\}^{\s^2} $ the space of expansions of $\s$ whose $\mathrm{Age}$ is exactly $\mathcal{S}^*$. We will denote $E=(<^E,R^E)$ the elements of $\mathcal{M}$, by identification with the structure that can be inferred from the expansion.
The result shown in \cite{JLNVTW} is:
\begin{theorem}\label{thm2}
The universal minimal flow of $\Aut(\s)$ is $ \Aut(\s)\curvearrowright \mathcal{M}$.
\end{theorem}

We are interested in showing that the $\mathrm{Aut}(\s)$-invariant measures on $\mathcal{M}$ are all equal. A useful tool of measure theory is the following Lemma (see \cite{Gut} Theorem $3.5$)

\begin{lemma}\label{Lem:UniMeas}
Let $\mu$ and $\nu$ be two probability measures defined on a $\sigma$-field $\mathcal{E}$. If there is a family $(A_n)_{n\in \N}\in \mathcal{E}^\N$ stable under intersection that generates $\mathcal{E}$ and such that for all $n\in \N$, $\mu(A_n)=\nu(A_n)$, then $\mu=\nu$.
\end{lemma}
The rest of this section is devoted  to describing a family $\mathcal{P}$ of clopen sets that generate the Borel sets of $\mathcal{M}$. The sets of our family $\mathcal{P}$ are of the form


\begin{equation*}
U_{(x_i)_{i=1}^n, (\varepsilon_i^j)_{1\leq i <j \leq n}} \cap V_{(a^1_1,\ldots, a^1_{i_1}),\ldots,(a^k_1,\ldots, a^k_{i_k})}\subset \mathcal{M}.
\end{equation*}
They are defined as follows.

Let $(x_i)_{i=1}^n$ be in different columns. Let $(\varepsilon_i^j)_{i<j\leq n}\in \{0,1\}^{n\choose 2}$. An element $E=(<^E,R^E)\in \mathcal M$ belongs to $U_{(x_i)_{i=1}^n, (\varepsilon_i^j)_{1\leq i <j \leq n}}$ iff the following conditions are satisfied :
\begin{enumerate}
\item   $(x_1^\perp <^E\ldots <^E x_n^\perp)$
\item for $k<l$, 
\[
R^E(x_k,x_l) \Leftrightarrow (x_k \rightarrow x_l)^{\varepsilon_k^l}.
\]
where for all $x,y\in \s$ and $\varepsilon \in \{0,1\}$, $(x\rightarrow y)^\varepsilon$ means $(x\rightarrow y)$ if $\varepsilon=1$ and $\neg (x\rightarrow y)$ otherwise.
\end{enumerate}


The rest of $R$ on those columns can be recovered from this by construction of $\mathcal{S}^*$. Indeed, observe that for all $x\in x_k^\perp$, $y\in x_l^\perp$, we have
\begin{align*}
R^E(x,y) &\Leftrightarrow \left( \left( x\sim_{x_l^\perp}^{\s} x_k \text{ and } R^E(x_k,x_l)\right)
\text{ or }
\left( x\nsim_{x_l^\perp}^{\s} x_k \text{ and } \neg R^E(x_k,x_l)\right)
 \right)
\end{align*}
and
$$ R^E(y,x) \Leftrightarrow \left( \left(y\rightarrow x_k \text{ and } R^E(x_k,y)\right) \text{ or } \left(x_k\rightarrow y \text{ and } \neg R^E(x_k,y)\right)  \right).$$

An important remark is that if we have a different family $(x_1',\ldots,x_n')$ such that $x_i\perp x_i'$, then there is a family $(\alpha_i^j)_{1\leq i <j\leq n}$ such that
\[ U_{(x_i)_{i=1}^n, (\varepsilon_i^j)_{1\leq i <j \leq n}} =U_{(x_i')_{i=1}^n, (\alpha_i^j)_{1\leq i <j \leq n}} . \]
This can be achieved by taking $\alpha_i^j = \varepsilon_i^j$ if $x_i \sim_{x_j^\perp} x_i'$ and $\alpha_i^j = 1 - \varepsilon_i^j$ otherwise.

An additional remark that will be useful throughout the paper is that for a given family $(x_1,\ldots, x_n)$ of elements taken in different columns,
\[\mathcal{M} = \bigsqcup_{\sigma\in S_n, (\varepsilon_i^j)_{1\leq i < j \leq n}}U_{(x_{\sigma(i)})_{i=1}^n, (\varepsilon_i^j)_{1\leq i <j \leq n}}.
\]

We also define
\begin{align*}
V_{(a^1_1,\ldots, a^1_{i_1}) , \ldots, (a^k_1,\ldots, a^k_{i_k})} =  \{E&\in \mathcal{M}  : \\
 &(a^1_1<^E\cdots <^Ea^1_{i_1})\wedge\cdots\wedge (a^k_1<^E\cdots <^E a^k_{i_k}) \}
\end{align*}
where $(a_i^j\perp a_{i'}^{j'})$ iff $j=j'$.

This collection of sets is a generating family for the open sets of our space, so it is also a generating family for the Borel sets. 

To use Lemma \ref{Lem:UniMeas}, we would also need to know that this family is stable under intersection, unfortunately this is not the case. 
However, the intersection of two sets in $\mathcal{P}$ is actually a disjoint union of sets in $\p$. Therefore if we consider $\mathcal{P}'$ the collection of finite intersection of elements of $\p$, the evaluation of a measure on an element of $\p'$ is determined by the evaluation of the measure on $\p$. By Lemma \ref{Lem:UniMeas}, any measure is entirely characterized by its evaluation on elements of $\p'$, so it is characterized by its evaluation on elements of $\p$.

\section{Invariant measures}

From this point on, we denote $G=\Aut(\s)$. Let us first define $\mu_0$ a $G$-invariant probability measure on $\mathcal{M}$. We define $\mu_0$ by:

\begin{equation*}
\mu_0 \left(U_{(x_i)_{i=1}^n, (\varepsilon_i^j)_{1\leq i <j \leq n}} \cap V_{(a^1_1,\ldots, a^1_{i_1}),\ldots,(a^k_1,\ldots, a^k_{i_k})} \right) 
= \frac{1}{n! 2^{\binom{n}{2}}}\frac{1}{\displaystyle\prod_{j=1}^k i_j !}.
\end{equation*}
We call $\mu_0$ the uniform measure. It is proven in \cite{SP} that this measure is well-defined on all Borel sets and that it is $G$-invariant. We want to show that it is actually the only invariant measure. By Lemma \ref{Lem:UniMeas}, we only have to check that the invariant measures coincide on  $U_{(x_i)_{i=1}^n, (\varepsilon_i^j)_{1\leq i <j \leq n}} \cap V_{(a^1_1,\ldots, a^1_{i_1}),\ldots,(a^k_1,\ldots, a^k_{i_k})} $.

Before proving Theorem \ref{thm1} we need to prove the following preliminary results:

\begin{prop}\label{Prop1} For all $(x_i)_{i=1}^n$ such that $\neg (x_i \perp x_j)$ for $i\neq j$ and $(\varepsilon_i^j)_{i< j \leq n}\in 2^{n \choose 2}$, we have:
$$\mu\left(U_{(x_i)_{i=1}^n, (\varepsilon_i^j)_{1\leq i <j \leq n}}\right) = \frac{1}{n! 2^{\binom{n}{2}}}.$$
\end{prop} 
\begin{prop}\label{Prop2}For all $(a_1^1,\ldots, a_{i_1}^1,\ldots,a_1^k,\ldots, a_{i_k}^k)$ such that $a_i^j\perp a_{i'}^{j'}$ iff $j=j'$, we have:
$$\mu\left( V_{(a^1_1,\ldots, a^1_{i_1}),\ldots,(a^k_1,\ldots, a^k_{i_k})} \right) = \frac{1}{\displaystyle\prod_{j=1}^k i_j !}.$$

\end{prop}
Similar results were proven in \cite{SP}. We will prove those results using different methods. The proof of Proposition \ref{Prop2} is very similar to what we will do later on in order to conclude and contains the key argument of this paper.

For proofs of Proposition \ref{Prop2} and Theorem \ref{thm1}, we will need an ergodic decomposition theorem, thus we need to define the notion of ergodicity.

\begin{defn}Let $\Gamma$ be a Polish group acting continuously on a compact space $X$. A $\Gamma$-invariant measure $\nu$ is said to be $\Gamma$-\emph{ergodic} if for all $A$ measurable such that 
$$ \forall g\in \Gamma \ \nu(A \triangle g\cdot A)=0,$$
we have $\nu(A)\in \{0,1\}$.
\end{defn}

We can now state the following (see \cite{Phelps} Proposition $12.4$):
\begin{theorem}\label{thm3}
Let $\Gamma$ be a Polish group acting continuously on a compact space $X$. Let $P(X)$ denote the space of probability measures on X and $P_\Gamma (X)=\{ \mu \in P(X) \ : \  \Gamma \cdot \mu = \mu \}$. Then, the extreme points of $P_\Gamma (X)$ are the $\Gamma$-ergodic invariant measures.
\end{theorem}

We will also need to use Neumann's Lemma (see \cite{Cam}, Theorem $6.2$) :
\begin{theorem} \label{Thm:Neu} Let $H$ be a group acting on $\Omega$ with no finite orbit. Let $\Gamma$ and $\Delta$ be finite subsets of $\Omega$, then there is $h\in H$ such that $h\cdot \Gamma \cap \Delta = \varnothing.$
\end{theorem}

The remaining of the section will be divided in three subsection. One for the proof of Proposition \ref{Prop1}, one for the proof of Proposition \ref{Prop2} and finally one for the proof of Theorem \ref{thm1}.

\subsection{Proof of Proposition \ref{Prop1}}

For this proof, we will need the following technical lemma.
\begin{lemma}\label{Lemma1}
Let $k< n$, let $P_1, \ldots, P_n$ be different columns in $\s$ and  let $y_1\in P_1 ,\ldots,y_k \in  P_k$. Take a given family $\varepsilon_i^j \in \{0,1\}$ where $1\leq i<j \leq n$ and $k<j$. Then there exist $y_{k+1}\in P_{k+1},\ldots,y_n  \in P_n$ such that $(y_i \rightarrow y_j)$ iff $\varepsilon_i^j=1$ for all $i<j$ and $k<j$.
\end{lemma}
\begin{proof}
Take $x_{k+1}\in P_{k+1},\ldots,x_n\in P_n$. Consider the following structure 
$$\mathbf{A}=((y_1^A,\ldots,y_n^A,x^A_{k+1}, \ldots, x_n^A),\to^\mathbf{A})$$
where $(y_i^A \to^\mathbf{A} y_j^A) \Leftrightarrow (y_i\to y_j)$ if $i<j\leq k$, $(y_i^A \to^\mathbf{A} y_j^A) \iff (\varepsilon_i^j =1)$ if $1\leq i <j \leq n$ and $k<j$. We also have $x^A_i \perp^\mathbf{A} y_i^A$ for $i>k$ and $(x_i^A \to^\mathbf{A} x_j^A \Leftrightarrow x_i \to x_j)$ for $k<i<j$. 

We put edges between $x_i^A$ and $y_j^A$ in order for them to respect the parity condition. Remark that there is more than one way to do this, for instance one can ask that when $k<i<j$, $(x_i^A \to^\mathbf{A} y_j^A) \Leftrightarrow (x_i^A \to^\mathbf{A} x_j^A)$ and $(x_j^A \to^\mathbf{A} y_i^A) \Leftrightarrow (y_j^A \to^\mathbf{A} y_i^A)$. The remaining edges can be added arbitrarily because they concern columns with only one vertex.

We make sure that $\mathbf{A}\in \mathcal{S}$. Indeed, noting that since there is one point in the first $k$ columns, and two in the remaining ones, it suffices to check the parity condition in the last $n-k$ columns. Take $k<j<i\leq n$. We know that the edges between $x_i^A$ and $y_j^A$ and the edge between $x_i^A$ and $x_j^A$ go in the same direction. Similarly, the edge between $x_j^A$ and $y_i^A$ and the edge between $y_j^A$ and $y_i^A$ also go in the same direction. Therefore the parity condition must be respected.

Remark that $((y^A_1,\ldots,y^A_k,x^A_{k+1},\ldots,x^A_n),\to^\mathbf{A})$ and $((y_1,\ldots,y_k,x_{k+1},\ldots,x_n),\to^\s)$ are isomorphic, hence $\mathbf{A}$ embeds in $\s$ in a way that extends this isomorphism. The image of $(y_{k+1}^A,\ldots,y_n^A)$ is as wanted.

\end{proof}

The fundamental observation for the proof of Proposition \ref{Prop1} is that if we take $x_1,\ldots,x_n\in \s$ all in different columns,
\[
\overline{\Aut(\s)\cdot (<^*,R^*)}=\underset{\sigma \in S_n, \ (\varepsilon_i^j)_{1\leq i <j \leq n}}\bigsqcup U_{(x_{\sigma(i)})_{i=1}^n,(\varepsilon_i^j)_{1\leq i <j \leq n}}.
\]

We will show that for any two families $\varepsilon= \left(\varepsilon_i^j\right)_{i< j\leq n} $, $\alpha=\left( \alpha_i^j\right)_{i < j\leq n} $ and $\sigma \in S_n$ there is a $g\in G$ such that 

$$U_{(x_i)_{i=1}^n,\varepsilon}=  g\cdot U_{(x_{\sigma(i)})_{i=1}^n,\alpha }. $$

This means that all sets of this form have the same measure, hence we will have the result because there are $n! 2^{\binom{n}{2}}$ such sets.

First, we construct $g'\in G$ such that 
\begin{equation*}
g'\cdot U_{(x_{\sigma(i)})_{i=1}^n,\alpha} = U_{x_1,\ldots,x_n,\beta}
\end{equation*} 
for some $\beta=(\beta_i^j)_{1\leq i< j\leq n}$.

We  want to prove that there is $g'\in G$ such that $g'\cdot x_i \in (x_{\sigma(i)})^\perp$. By Lemma \ref{Lemma1}, there exists $x_1',\ldots, x_n' \in \mathbb{S}$ such that $x_{\sigma(i)} \perp x_{i}'$ and $x_i \to x_j$ iff $x_i' \to x_j'$. Remark that by construction, there is a partial automorphism $\tau$ that sends $x_{\sigma(i)}$ to $x_{i}'$.
By homogeneity, there is $g'$ an automorphism of $\s$ that extends $\tau$. We remark that 
\begin{equation*}
g' \cdot U_{(x_i)_{i=1}^n,\alpha} = U_{(x'_{\sigma(i)})_{i=1}^n,\alpha}
\end{equation*}
and as we observed before, $U_{(x'_{\sigma(i)})_{i=1}^n,\alpha}$ does not depend on $x_i'$, but on their columns. Thus, there exist a family $\beta=(\beta_i^j)_{1\leq i <j \leq n}$ such that 
\begin{equation*}
U_{(x_{\sigma(i)}')_{i=1}^n,\alpha} = U_{(x_i)_{i=1}^n,\beta}.
\end{equation*}

Next, we construct $h\in G$ such that 
$$U_{(x_i)_{i=1}^n,\varepsilon}=  h\cdot U_{(x_i)_{i=1}^n, \beta }. $$

Assume that there are $k<l$ such that $\beta_i^j=\varepsilon_i^j$ for all $(i,j) \neq (k,l)$ and $\beta_k^l \neq \varepsilon_k^l$. Remark that taking care of this case will be enough to prove the result : If $\alpha$ and $\beta$ disagree in more than one coordinate, iterating this process still allows to modify coordinates one at a time.

Let us take $x_k' \perp x_k$ such that for all $ i \in [n]\backslash \{k,l\}$, $x_k' \rightarrow x_i$ iff $x_k \rightarrow x_i$ and $x_k' \rightarrow x_l$ iff $x_l \rightarrow x_k$. This is possible using Lemma \ref{Lemma1} where $\{y_1,\ldots,y_{n-1}\}=\{x_1,\ldots,x_n\}\backslash\{x_k\}$ and $P_n=x_k^\perp$. We define $x_l'\perp x_l$ similarly.

We take $h\in G$ such that $h(x_i)=x_i$ for all $ i \in [n]\backslash \{k,l\}$, $h(x_k')=x_k$ and $h(x_l')=x_l$. By homogeneity, such a $h$ exists: indeed, by the parity condition, we have $(x_k\rightarrow x_l)\Leftrightarrow (x'_k \rightarrow x'_l)$. Let us prove that $h$ gives the result. 

Take $E\in U_{x_1,\ldots,x_n, \beta }$. We will prove that 
$$h\cdot E \in U_{(x_i)_{i=1}^n,\varepsilon}.$$ 
For all $i<j$ we want to prove that
$$ R^{h\cdot E}(x_i,x_j)  \Leftrightarrow (x_i \rightarrow x_j)^{\varepsilon_i^j},$$
and since 
$$R^{h\cdot E}(x_i,x_j)\Leftrightarrow R^E (h^{-1}( x_i) , h^{-1}( x_j) ),$$
we prove
$$R^E (h^{-1}( x_i) , h^{-1}( x_j) ) \Leftrightarrow (x_i \rightarrow x_j)^{\varepsilon_i^j}.$$

If $\{i,j\}\cap \{k,l\} = \varnothing$, the result is obvious.

If $j=k$ and $i<k$, we have:
\begin{align*}
R^{h\cdot E} (x_i,x_k) &\Leftrightarrow R^E(h^{-1}(x_i),h^{-1}(x_k)) \\
&\Leftrightarrow (x_i \rightarrow h^{-1}(x_k))^{\beta_i^k} \\
&\Leftrightarrow (x_i \rightarrow x_k')^{\beta_i^k} \\
&\Leftrightarrow (x_i \rightarrow x_k)^{\beta_i^k},
\end{align*}
and since $\beta_i^k=\varepsilon_i^k$, we have
$$R^{h\cdot E} (x_i,x_k) \Leftrightarrow (x_i \rightarrow x_k)^{\varepsilon_i^k}.$$

The other cases where $|\{i,j\}\cap \{k,l\}| =1$ are similar.

Finally, if $(i,j)=(k,l)$, we have:
\begin{align*}
R^{h\cdot E} (x_k,x_l) &\Leftrightarrow R^E(h^{-1}(x_k),h^{-1}(x_l)) \\
&\Leftrightarrow (x_k \rightarrow h^{-1}(x_l))^{\beta_k^l} \\
&\Leftrightarrow (x_k \rightarrow x_l')^{\beta_k^l} \\
&\Leftrightarrow (x_k \rightarrow x_l)^{\varepsilon_k^l}.
\end{align*}
The last equivalence is a direct consequence of the definition of $x_l'$ and the fact that $\beta_k^l=(1-\varepsilon_k^l)$.
\qed

\subsection{Proof of Proposition \ref{Prop2}}
We prove the result by induction on the number $k$ of columns.

By homogeneity, for any column $(a_1^j)^\perp$ and $\sigma \in S_{i_j}$ there exists $g\in G$ such that 
\[ g\cdot V_{(a^j_1,\ldots, a^j_{i_j})} = V_{(a^j_{\sigma(1)},\ldots, a^j_{\sigma(i_j)})},
\]
thus 
\[ \mu \left( V_{(a^j_1,\ldots, a^j_{i_j})} \right) = \frac{1}{i_j !}.\]
This proves the initial case. 

Let us now assume that for all $(a_1^1,\ldots, a_{i_1}^1,\ldots,a_1^{k-1},\ldots, a_{i_{k-1}}^{k-1})$ such that $a_i^j\perp a_{i'}^{j'}$ iff $j=j'$, we have
$$\mu\left( V_{(a^1_1,\ldots, a^1_{i_1}),\ldots,(a^{k-1}_1,\ldots, a^{k-1}_{i_{k-1}})} \right) = \frac{1}{\displaystyle\prod_{j=1}^{k-1} i_j !}.$$

We consider $(a_1^k,\ldots,a_{i_k}^k)$ all in the same column and not in any $(a_1^i)^\perp$ for $i<k$. Remark that 
\[ V_{(a^1_1,\ldots, a^1_{i_1}),\ldots,(a^{k}_1,\ldots, a^{k}_{i_{k}})}  = V_{(a^1_1,\ldots, a^1_{i_1}),\ldots,(a^{k-1}_1,\ldots, a^{k-1}_{i_{k-1}})} \cap V_{(a_1^k,\ldots,a_{i_k}^k)}.\]
We want to prove that the ordering of $(a_1^k)^\perp$ is independent from the ordering of the other columns.

 Enumerate as $(V_1,\ldots , V_\tau)$ all the different sets of the form $V_{(a^1_{\sigma_1(1)},\ldots, a^1_{\sigma_1(i_1)}),\ldots,(a^{k-1}_{\sigma_{k-1}(1)},\ldots, a^{k-1}_{\sigma_{k-1}({i_{k-1}})})}$ where $\sigma_j$ is a permutation of $\{1,\ldots,i_j\}$. Thus $\tau = \displaystyle \prod_{j=1}^{k-1} i_j !$. 

For all $l\in \{1,\ldots, \tau \}$, we define
$$\mu_{V_l} (\cdot) = \frac{\mu(\cdot \cap V_l)}{\mu(V_l)}.$$
This is the conditional probability of $\mu$ given $V_l$. We remark that:
$$\mu = \sum_{l=1}^\tau \mu(V_l) \mu_{V_l}. $$

Denote $\mathrm{LO}((a_1^k)^\perp )$ the space of linear orderings on $(a_1^k)^\perp$. There is a restriction map $r$ from $\mathcal{M}$ to $\mathrm{LO}((a_1^k)^\perp )$. We denote $V^r_{(a_1^k,\ldots,a_{i_k}^k)}$ the image of $V_{(a_1^k,\ldots,a_{i_k}^k)}$ by $r$.
Let $\nu$ be, the pushforward of $\mu$ on $\mathrm{LO}({a_1^1}^\perp$) by $r$, and let $\nu_{V_l}$ be the pushforward of $\mu_{V_l}$ by the same map. We have:
$$\nu = \sum_{l=1}^\tau \mu(V_l) \nu_{V_l}. $$
Observe that the initial step of the induction implies that $\nu$ is the uniform measure on $\mathrm{LO}((a_1^k)^\perp )$

We denote $\mathrm{Stab}^{\text{set}}_{(a_1^k)^\perp}$ the setwise stabilizer of $(a_1^k)^\perp$, $\mathrm{Stab}^{\text{pw}}_{(a^1_1,\ldots,a^1_{i_1},\ldots ,a^{k-1}_1,\ldots,a^{k-1}_{i_{k-1}})}$ the pointwise stabilizer of $(a^1_1,\ldots,a^1_{i_1},\ldots ,a^{k-1}_1,\ldots,a^{k-1}_{i_{k-1}})$ and set $H=\mathrm{Stab}^{\text{set}}_{(a_1^k)^\perp} \cap \mathrm{Stab}^{\text{pw}}_{(a^1_1,\ldots,a^1_{i_1},\ldots ,a^{k-1}_1,\ldots,a^{k-1}_{i_{k-1}})}$. We remark that $\nu_{V_l}$ is $H$-invariant for all $l\in \{1,\ldots,\tau\}$.

Since $\mathrm{LO}({a_1^1}^\perp)$ is compact, by Theorem \ref{thm3}, if we prove that $\nu$ is $H$-ergodic, then we have the result. Indeed, then $\nu$ is an extreme point of the $H$-invariant measures and all the $\nu_{V_l}$ are equal to $\nu$, thus for any $l$ we have
\begin{align*}
\mu\left(V_{(a^k_1,\ldots, a^k_{i_k})} \cap V_l\right) &= \mu_{V_l} \left(V_{(a^k_1,\ldots, a^1_{i_k})}\right) \mu(V_l)
\\ &= \nu_{V_l} \left(V^r_{(a^k_1,\ldots, a^1_{i_k})}\right) \mu(V_l)
\\ &= \nu \left(V^r_{(a^1_k,\ldots, a^1_{i_k})}\right) \mu(V_l)
\\ &= \frac{1}{i_k !} \frac{1}{\displaystyle\prod_{j=1}^{k-1} i_j !}
\end{align*}
and this equality finishes the induction.

It only remains to prove the ergodicity of $\nu$. The following lemma will allow us to conclude.

\begin{lemma} \label{Lem:ErgoInf} Let $K$ be a group acting on a set $\mathcal{N}$ with no finite orbits. Denote $\mathrm{LO}(\mathcal{N})$ the space of linear orderings on $\mathcal{N}$. Then the uniform measure $\lambda$ on $\mathrm{LO}(\mathcal{N})$ is $K$-ergodic. 
\end{lemma}
\begin{proof}

Take $A$ a Borel subset of $\mathrm{LO}(\mathcal{N})$ such that for all $g\in K$ $\lambda(A\triangle g\cdot A)=0$. Let $\varepsilon >0$. There is a cylinder, i.e. a set depending only on a finite set of $\mathcal{N}$, $B=B(b_1,\ldots , b_k)$ such that $\mu(B\triangle A ) \leq \varepsilon$. 
Using Neumann's Lemma, we get that there exists $g\in K$ such that
$ \{b_1,\ldots,b_k\} \cap g\cdot \{b_1,\ldots,b_k\} = \varnothing$.

Moreover, since $\nu$ is uniform, the orderings of two disjoint sets of points are independent. Indeed, taking $(a_1,\ldots, a_i)$ and $(c_1,\ldots, c_{i'})$ two disjoint families of points. Note that  $\lambda(a_1<\cdots < a_i \cap c_1<\cdots < c_{i'} )  $ is equal to the number of way to insert $(c_1,\ldots, c_{i'})$ in $(a_1,\ldots,a_i)$ respecting both orderings times the weight of a given ordering of $(a_1,\ldots, a_i, c_1,\ldots,c_{i'})$. We therefore have

\begin{equation*}
\begin{split}
\lambda(a_1<\cdots < a_i \cap c_1<\cdots < c_{i'} )  
 &= \binom{i+i'}{i} \frac{1}{(i+i')!}
\\ &= \frac{1}{i!}\frac{1}{i'!}.
\end{split}
\end{equation*}

This means that $B$ and $g\cdot B$ are independent. We can now write:
\begin{align*}
\left| \lambda(A) - \lambda(A)^2 \right| = & \left| \lambda(A \cap g\cdot A) - \lambda(A)^2 \right| \\
\leq & \left| \lambda(A\cap g\cdot A) - \lambda(B\cap g\cdot A) \right| + \left| \lambda(B \cap g\cdot A) - \lambda(B\cap g\cdot B) \right| \\
& + \left| \lambda(B\cap g\cdot B) - \lambda(B)^2 \right| + \left| \lambda(B)^2 - \nu(A)^2 \right| \\
\leq& 4\varepsilon.
\end{align*}
The last inequality comes from the following inequalities 
\begin{align*}
&\left| \lambda(A\cap g\cdot A) - \lambda(B\cap g\cdot A) \right| \leq \lambda((A \triangle B) \cap g\cdot A) \leq \varepsilon , \\
&\left| \lambda(B \cap g\cdot A) - \lambda(B\cap g\cdot B) \right| \leq \lambda(g\cdot(A \triangle B) \cap B) \leq \varepsilon ,\\
&\lambda(B\cap g\cdot B) = \lambda(B)^2 \\
&\text{and} \\
&\left| \lambda(B)^2 - \lambda(A)^2 \right| = (\lambda(A)+\lambda(B))|\lambda(A)-\lambda(B)| \leq 2 \varepsilon.
\end{align*}
This proves that $\lambda$ is $K$-ergodic.

\end{proof}

We only have to prove that $H$ has no finite orbits on $(a_1^1)^\perp$. It suffices to remark that for all $a\in \s$, $(u_1,\ldots, u_i)\in \s$, there are infinitely many $b\in a^\perp$ such that $a\rightarrow u_j$ iff $b\rightarrow u_j$ for all $1\leq j\leq i$. 

Indeed, take $k\in \N$. Consider the structure $((a_1,\ldots,a_k,v_1,\ldots , v_i),\to)$, where $a_l\perp a_j$, $a_l \to v_k$ iff $a \to u_k$ and $v_m \to v_{m'}$ iff $u_m\to u_{m'}$ for all $l,j\leq k$ and $ m,m'\leq i$. It is obvious that this structure verifies the parity condition. Therefore in $\s$ we can find $k$ copies of $a$ in its column for any $k>0$.

This is enough to conclude that $\nu$ is indeed $H$-ergodic.

\qed

\subsection{Proof of Theorem \ref{thm1}}

In what follows, we will show that 

$$ \mu\left(U \cap V \right) = \mu\left(U\right) \mu\left( V \right) $$
for all $U=U_{(x_i)_{i=1}^n,  (\varepsilon_i^j)_{1\leq i <j \leq n}}$ and $V= V_{(a^1_1,\ldots, a^1_{i_1}),\ldots,(a^k_1,\ldots, a^k_{i_k})}$. It will follow that $\mu=\mu_0$.

Let us take a certain set $\{x_1,\ldots,x_n\}$ where none of the $x_i$ are in the same column. We denote $m$ the number of sets $U$ as above associated to this family. 
We consider $(U_i)_{i=1}^{m}$ the disjoint sets of $\mathcal{M}$ corresponding to the ways of defining a relation $R$ and an order on the columns $x_1^\perp,\ldots, x_n^\perp$, i.e. $U_i = U_{(x_{\sigma(i)})_{i=1}^n, \varepsilon}$ for some $\sigma\in S_n$ and $\varepsilon \in 2^{n \choose 2}$. Proposition \ref{Prop1} tells us that:
$$  \forall i,j \in \{1,\ldots, m \}, \ \mu (U_i) = \mu (U_j) .$$

We remark that this quantity is $\frac{1}{m}$. We now define, for all $i \in \{1,\ldots, m \}$,
\begin{equation*}
\mu_{U_i}(\cdot) = \frac{\mu\left(\cdot \cap U_i\right)}{\mu\left(U_i\right)}.
\end{equation*}

This is the conditional probability of $\mu$ given $U_i$. Denote $H$ the subgroup of $G$ that stabilizes $x_i^\perp$ for all $1\leq i \leq n$ and each $\sim_{x_j^\perp}$-equivalence class in $x_i^\perp$ for $i\neq j$. Remark that $H$ stabilizes $U_i$, by construction, hence $\mu_{U_I}$ is $H$-invariant.

A simple but fundamental remark is that since $\displaystyle \bigsqcup_{i=1}^m U_i = \mathcal{M}$ and all the $U_i$ have the same measure under $\mu$, we have

$$ \mu = \frac{1}{m} \sum_{i=1}^m \mu_{ U_i}. $$

Let $\mathrm{LO}_p(\s)$ denote the space of partial orders that are total on each column and do not compare elements of different columns. There is a restriction map from $\mathcal{M}$ to $\mathrm{LO}_p(\s)$. We consider $\lambda$ the pushfoward of $\mu$ on $\mathrm{LO}_p(\s)$ by this map. Similarly, we consider $\lambda_{U_i}$ the pushfoward of $\mu_{U_i}$ on $\mathrm{LO}_p(\s)$. We have
$$ \lambda = \frac{1}{m} \sum_{i=1}^m \lambda_{ U_i}. $$

The rest of the proof is similar to the proof of Proposition \ref{Prop2}: we prove that $\lambda$ is $H$-ergodic. Take $A$ a Borel subset of $\mathrm{LO}_p(\s)$ such that for all $h\in H$, $\lambda(A \triangle h\cdot A)=0$. For any $\varepsilon > 0$, there is a cylinder $B$ that depends only on finitely many points $(b_1,\ldots,b_k)$ such that $\lambda(A\triangle B)\leq \varepsilon$. We now want to find an element $g\in H$ such that  $B$ and $g\cdot B$ are $\lambda$-independent.

Take $\mathbf{A}\in \mathcal{S}$, then consider $\mathbf{A}'$ as a structure that contains disjoints copies of $\mathbf{A}$ that we call $\mathbf{A}_1$ and $\mathbf{A}_2$. We impose that all edges between elements of $\mathbf{A}_1$ and elements of $\mathbf{A}_2$ go from $\mathbf{A}_1$ to $\mathbf{A}_2$. Necessarily, $\mathbf{A}'\in \mathcal{S}$, so in $\s$ there are copies of any finite substructure that lives in a disjoint set of columns. Therefore, there is an element $g\in G$ such that $ (b_1,\ldots , b_k ) $ and $g\cdot ( b_1, \ldots,b_k)$ are in disjoint sets of columns. It is easy to see by Proposition \ref{Prop2} that $B$ and $g\cdot B$ are $\lambda$-independent.

 Just as in the proof of Proposition \ref{Prop2}, we have:
\begin{align*}
\left| \lambda(A) - \lambda(A)^2 \right| = & \left| \lambda(A \cap g\cdot A) - \lambda(A)^2 \right| \\
\leq & \left| \lambda(A\cap g\cdot A) - \lambda(B\cap g\cdot A) \right| + \left| \lambda(B \cap g\cdot A) - \lambda(B\cap g\cdot B) \right| \\
& + \left| \lambda(B\cap g\cdot B) - \lambda(B)^2 \right| + \left| \lambda(B)^2 - \lambda(A)^2 \right| \\
\leq& 4\varepsilon.
\end{align*}
Thus $\lambda(A)\in \{0,1\}$.

Since $\mathrm{LO}_p(\s)$ is compact, we have the result: $\lambda$ is an extreme point of the $H$-invariant measures and all the $\lambda_{U_i}$ are equal. Therefore we have,
\begin{align*}
\mu(V \cap U_i)&= \mu_{U_i}(V) \mu(U_i)
\\ &= \lambda_{U_i}(V) \mu(U_i)
\\ &= \lambda(V) \mu(U_i)
\\&=\mu(V) \mu(U_i)
\end{align*}
for all $i\in \{1,\ldots,m\}$, and $V=V_{(a^1_1,\ldots, a^1_{i_1}),\ldots,(a^k_1,\ldots, a^k_{i_k})}$. This finishes the proof of Theorem \ref{thm1}.

\bibliographystyle{amsalpha}
\bibliography{Semigeneric}

\end{document}